\documentclass[11pt,a4paper,reqno]{amsart}

\usepackage[T1]{fontenc}
\usepackage[utf8]{inputenc} 
\usepackage[french,english]{babel}
\usepackage{lmodern} 
\usepackage{amsmath, amssymb, amsfonts, amsthm}
\usepackage{graphicx}
\usepackage{tabvar} 
\usepackage{multicol} 
\usepackage{geometry} 
\usepackage{hyperref} 
\usepackage{xcolor} 
\usepackage{fancyhdr} 

\usepackage{float}
\usepackage{cite}
\usepackage{enumerate}     

\hypersetup{
  colorlinks = true,
  linkcolor = blue,
  citecolor = magenta,
  urlcolor  = teal,
  pdftitle  = {Research Article},
  pdfauthor = {Passimzouwé Dagou, Pagdame Tiebekabe, Kokou Tchariè},
}

\newtheorem{proof.}{Proof}
\newtheorem{theorem}{Theorem}

\newtheorem{lem}{Lemma}
\newtheorem{step}{Step}
\newtheorem{def.}{Definition}[section]

\title{On Narayana numbers which are products of four $b$-repdigits with a consequence}

\author{Passimzouwé Dagou}
\address{Department of Mathematics, University of Kara, Togo}
\email{\url{passimzouwedagou@gmail.com}}

\author{Pagdame Tiebekabe}
\address{Department of Mathematics, University of Kara, Togo}
\email{\url{tpagdame.math@gmail.com}}

\author{Kokou Tchariè}
\address{Department of Mathematics, University of Kara, Togo}
\email{\url{tkokou09@gmail.com}}

\begin{document}

\pagestyle{fancy}
\maketitle
\textbf{Abstract}. 
In this paper, we focus on Narayana numbers which can be written as a products of four repdigits in base $g$, where $g$ is an integer with $g\geq2$. We prove that for $g$ between $2$ and $12$, there are finitely many of these numbers. Moreover we have fully determined them.\\
\textbf{AMS Subject Classifications:} $11B39,\, 11J86,\, 11D61,\, 11Y50$.\\
\textbf{Keywords:} Diophantine equations; Narayana numbers; $g-$ Repdigits; Linear forms in logarithms; Reduction method. 
\section{Introduction}
\label{sect1}
In this paper we study an exponential Diophantine equation. A Diophantine equation is a polynomial equation, usually with two or more unknowns, for which only integer solutions are sought. They can be linear (sum of two or more monomials of degree 1) or exponential (the exponents on the terms can be unknown).\\
We will be specifically interested  in the Narayana numbers which can be written as products of four repdigits in base $g$ with $g\geq2.$\\

The Narayana sequence was introduced by the Indian mathematician Narayana Pandita in the 14th century in his treatise \textit{Ganita Kaumudi}. In this book, he poses a problem about the growth of cows, very similar to Fibonacci problem, but with a different maturity period. In this problem, it is assumed that a calf is born each year from every cow that is at least three years old. Narayana's cow's problem counts the number of calves produced each year \cite{ref4}.\\
The Narayana sequence is define by: $$ N_n=N_{n-1}+N_{n-3} \text{ for } n\geq 3 \text{ with } N_0=0,\; N_1=N_2=1.$$ 
In fixed base $g\geq2$, a repdigit has the following form: $$ \sum_{i=0}^{n-1}d\times g^i=d\times\frac{g^n-1}{g-1}, \text{ where } 1\leq d\leq g-1.$$ 
When $ g = 10 $, one usually omits to mention $ g $ and simply call these numbers as repdigits.\\
Several studies have been conducted  on determining  the terms of linear recurrence sequences that are repdigits in any base $g\geq2$. For more details refer to the recent  results \cite{ref1, ref3, ref6, ref13}. 
\\
Recently in \cite{ref14}, Adédji studied Mulatu and generalized Lucas numbers which are products of four $b-$repdigits and Pagdame et al. in \cite{ref11} Narayana numbers as products of three repdigits in base $g$ by basing their proof 
on double application of Baker's method and a reduction algorithm using computations based on continued fractions. In this paper, we will use the same approach for identifying  Narayana numbers that are products of four repdigits in base $g$.

The present paper is organized as follows. In Section \ref{sect2} we present the main results; Section \ref{sect3} recalls the necessary preliminaries; Section \ref{sect4} contains the proofs.

\section{Statement of main results}
\label{sect2}
In this section, we state all the main results obtained in this paper.
\begin{theorem}
	\label{theorm1}
	Let $g\geq 2$ be an integer. Then the Diophantine equation 
	\begin{equation}
	\label{eq1}
	N_k=d_1\frac{g^\ell-1}{g-1}\cdot d_2\frac{g^m-1}{g-1}\cdot d_3\frac{g^n-1}{g-1}\cdot d_4\frac{g^t-1}{g-1}
	\end{equation}
	has only finitely many solutions in integers $k, d_1, d_2, d_3, d_4, \ell, m, n, t$ such as $$1\leq d_i\leq g-1 \text{ for } i=1, 2, 3, 4 \text{ and } 1\leq \ell\leq m\leq n\leq t.$$ Furthermore, we have 
	$$
	t<2.11\times10^{67}\log^{12}g   \text{ and }  k< 2.54\times10^{68}\log^{13}g.
	$$
	In the following theorem, we completely and explicitly give all solutions of equation (\ref{eq1}) corresponding to $2\leq g\leq 12$.
\end{theorem}
\begin{theorem}
	\label{theorem2}
	$\{1,2,3,4,6,9,13,28,60,88,129,189\}$ is the set of only  Narayana numbers which are products of four repdigits in base g with $2\leq g  \leq 12$.\\
	Considering the notions introduced in the theorem \ref{theorm1}, the solutions to the equation (\ref{eq1}) can be expressed in the form: $$ N_k=[a,b,c,d]_g=a\times b\times c\times d \text{ in base } g$$
	where $$ a=d_1\frac{g^\ell-1}{g-1}, \, b=d_2\frac{g^m-1}{g-1},\,c=d_3\frac{g^n-1}{g-1},\, d=d_4\frac{g^t-1}{g-1}.$$ In 
	Therefore we have:

	Therefore, we have:
	\begin{center}
		\begin{table}[H]
			\caption{Narayana numbers that are products of four repdigits in base $g$, for $2 \leq g \leq 12$.}
			\label{table1}
			\vskip 0.5cm
			\begin{tabular}{|c|c||m{10cm}|}\hline 
				$k$ & $N_k$ & $[a,b,c,d]_g$ \\ \hline	
				$1,2,3$ & $1$ & $[1,1,1,1]_g$, for $g = 2,3,\dots,12$. \\ \hline
				$4$ & $2$ & $[1,1,1,2]_g$, for $g = 3,\dots,12$. \\ \hline
				$5$ & $3$ & $[1,1,1,11]_2$, $[1,1,1,3]_g$, for $g = 4,\dots,12$. \\ \hline
				$6$ & $4$ & $[1,1,1,11]_3$, $[1,1,1,3]_g$, for $g = 5,\dots,12$. \\ \hline
				$7$ & $6$ & $[1,1,1,11]_5$, $[1,1,2,3]_g$, for $g = 4,\dots,12$; $[1,1,1,6]_g$, for $g = 7,\dots,12$. \\ \hline
				$8$ & $9$ & $[1,1,11,11]_2$, $[1,1,1,11]_8$, $[1,1,1,9]_{10}$, $[1,1,3,3]_g$, for $g = 4,\dots,12$. \\ \hline
				$9$ & $13$ & $[1,1,1,111]_3$, $[1,1,1,11]_{12}$. \\ \hline
				$11$ & $28$ & $[1,1,1,44]_6$, $[1,1,2,22]_6$, $[1,1,4,11]_6$, $[1,1,4,7]_g$, for $g = 8,\dots,12$, $[1,2,2,7]_g$, for $g = 8,\dots,12$. \\ \hline
				
				$13$& $60$& \vskip 0.2cm $[1,2,2,33]_4, [1,2,3,22]_4, [2,2,3,5]_4, [1,1,1,66]_9,[1,1,2,33]_9,$ $[1,1,3,22]_9, [1,1,6,11]_9, [1,2,3,11]_9, [1,1,1,55]_{11},$ $ [1,1,5,11]_{11}, [1,1,6,A]_g, \text{ for } g=11,12, [1,2,3,A], \text{ for } g=11,12, [1,2,5,6]_g,\text{ for } g=7,\cdots,12, [1,3,4,5]_g,\text{ for } g=6,\cdots,12,[2,2,3,5]_g,\text{ for } g=6,\cdots,12.$ \vskip 0.2cm \\ \hline
				
				$14$& $88$& \begin{minipage}{10cm}\vskip 0.2cm
					$[1,1,1,88]_{10}, [1,1,8,11]_{g}, \text{ for } g=10, 12,$ \\ $ [1,1,2,44]_{10}, [1,1,4,22]_{10}, [1,2,4,11]_{g}, \text{ for } g=10, 12,$\\$ [1,2,2,22]_{10}, [2,2,2,11]_{g},\text{ for } g=10, 12.$ 	\end{minipage} \vskip 0.4cm\\ \hline
				$15$ & $129$ & $[1,1,1,333]_6$, $[1,1,3,111]_6$. \\ \hline
				$16$& $189$&\begin{minipage}{10cm}\vskip 0.2cm
					$[1,1,11,111111]_2, [11,11,11,111]_2, [1,1,3,333]_4,$ $ [1,3,3,111]_4, [1,3,3,33]_6, [3,3,3,11]_6,  [1,1,3,77]_8,$ $ [1,1,7,33]_8, [3,3,3,7]_g, \text{ for } g=8,\cdots,12$\\ $ [1,3,7,9]_g, \text{ for } g=8,10,11,12.$ \vskip 0.4 cm\end{minipage}\\ \hline
			\end{tabular}
		\end{table}
	\end{center}
	
	We recall that in the  table above, the letter \textbf{A} denote the value 10 in base $g=11,12$.
\end{theorem}
\section{Preliminaries}\label{sect3}
In this section, we will cover the essential concepts needed for the proofs of the results stated in section \ref{sect2}.
\subsection{Some properties of Narayana sequence}
\begin{def.}
	The Narayana sequence $(N_n)_{n\geq0}$ is defined by the third-order linear recurrence relation $N_n = N_{n-1} + N_{n-3}$ for $n \geq 3$, where the initial conditions are given by $N_0 = 0$ and $N_1 = N_2 = 1$.\\
	
	Its characteristic equation for $(N_n)_{n\geq0}$  is given by $$x^{3} - x^2 - 1=0.$$ Its characteristic polynomial is $f(x)=x^{3}-x^{2}-1$ have three roots, $\alpha, \beta$ and $\gamma$ which are given by:
	\begin{eqnarray*}
		\alpha &=&\frac{1}{3}\left(\sqrt[3]{\frac{1}{2}(29-3\sqrt{93})}+\sqrt[3]{\frac{1}{2}(3\sqrt{93}+29)}+1 \right),\\
		\beta &=&\frac{1}{3}-\frac{1}{6}\left( 1-i\sqrt{3}\right)\sqrt[3]{\frac{1}{2}(29-3\sqrt{93})}-\frac{1}{6}\left(1+i \sqrt{3}\right) \sqrt[3]{\frac{1}{2}(3\sqrt{93}+29)},\\  
		\gamma&=&\frac{1}{3}-\frac{1}{6}\left( 1+i\sqrt{3}\right)\sqrt[3]{\frac{1}{2}(29-3\sqrt{93})}-\frac{1}{6}\left(1-i \sqrt{3}\right) \sqrt[3]{\frac{1}{2}(3\sqrt{93}+29)}. \\
	\end{eqnarray*} 
	Noticing that it has a real zero $\alpha$($>1$) and two conjugate complex roots $ \beta $ and $\gamma$ with $|\beta|=|\gamma|<1$. We have $\alpha\approx 1.46557$. 
\end{def.}

\begin{lem}
	\label{lem1}
	For the sequence $(N_n)_{n\geq0}$, we have
	\begin{enumerate}[(a)]
		\item \begin{equation}
		\label{eq2}
		\alpha^{n-2} \leq N_n \leq\alpha^{n-1} \text{ for all } n \geq1.
		\end{equation}
		\item $(N_n)_{n\geq 0}$ satisfies the following "Binet-like" formula  
		\begin{equation}
		\label{eq3}
		N_n=a\alpha^{n} + b\beta^n + c\gamma^n \text{ for all } n \geq 0.
		\end{equation} From the three initial values of Narayana sequence, and using Vieta's theorem, one has: $$ a=\frac{\alpha^2}{\alpha^3+2},\; b=\frac{\beta^2}{\beta^3+2},\; \text{ and } c=\frac{\gamma^2}{\gamma^3+2}.$$
		\item  If we write $ \zeta_n =N_n-a\alpha^n=b\beta^{n} + b{\gamma}^{n} $, then 
		\begin{equation}
		\label{eq4}
		|\zeta_n| < \dfrac{1}{\alpha^{n/2}} \text{ for all }  n \geq 1 .
		\end{equation}
	\end{enumerate}
\end{lem}
\begin{proof}
	By  induction on $ n $ we simply prove (a). The proof of (b) can be found in reference \cite{ref15}. Finally, the proof of (c) follows from the triangle
	inequality and the fact that $ |\beta| = |\gamma| < 1 $. We leave the details to the reader. 
\end{proof}

Let $\mathbb{K}_{f}:=\mathbb{Q}(\alpha, \beta)$ be the splitting field of the polynomial $f$ over $\mathbb{Q}$.\\ Then $[\mathbb{K}_f:\mathbb{Q}]=6$. Furthermore, $[\mathbb{Q}(\alpha):\mathbb{Q}]=3$. The Galois group of $\mathbb{K}_f$ over $\mathbb{Q}$ is given by \begin{eqnarray*}
	\mathcal{G}_f:=Gal(\mathbb{K}_f/\mathbb{Q})&\cong&\{(1), (\alpha\beta), (\alpha\gamma), (\beta\gamma), (\alpha\beta\gamma), (\alpha\gamma\beta)\}\\
	&\cong& S_3.
\end{eqnarray*}  
Thus, we identify the automorphisms of $\mathcal{G}_f$ with the permutation of the zeros of the polynomial $f$. For example, the permutation $(\alpha\beta)$ corresponds to the automorphisms $\sigma_f: \alpha\to\beta, \beta\to\alpha, \gamma\to \gamma$.
\subsection{Linear form in logarithms}
This subsection opens with a brief overview of key facts about the logarithmic height of algebraic numbers.\\
For an algebraic number $ \eta $ of degree $ d $ over $ \mathbb{Q} $ and
minimal primitive polynomial over the integers
$$ f(X)\colon=a_0\prod_{i=1}^{d}(X-\eta^{(i)})\in\mathbb{Z}[X],$$
with positive leading coefficient $ a_0 $, we write $ h(\eta) $ for its logarithmic height, given by $$  h(\eta)\colon=\frac{1}{d}\left(\log a_ 0+\sum_{i=1}^{d}\log\left( \max\left( |\eta^{(i)}|, 1\right) \right)  \right). $$
In particular, if $\eta= p/q$ is a rational number with $ gcd(p, q) = 1 $ and $ q > 0 $, then $ h(\eta) =\log\max(|p|, q) $. Several properties of the logarithmic height function $h(\eta)$, which will be used throughout the following sections without further mention, are also well established: \begin{eqnarray*}
	h(\eta\pm\gamma)&\leq& h(\eta)+h(\gamma)+\log2,\\
	h(\eta\gamma^{\pm1})&\leq& h(\eta)+h(\gamma),\\
	h(\eta^{s})&=&|s|h(\eta)\;\;\;(s\in\mathbb{Z}).
\end{eqnarray*}
In particular, if $\eta=p/q\in\mathbb{Q}$ is rational number in its reduced form with $q>0$, then $h(\eta)=\log(\max\{|p|, |q|\}).$
\begin{theorem}
	\label{theorem3}
	Let $ \mathbb{K} $ be a number field of degree $d_{\mathbb{K}}$ over $ \mathbb{Q} $, $ \eta_1,\cdots, \eta_s $ be positive real numbers of $\mathbb{K}$, and $ b_1,\cdots, b_s $ rational integers. Put
	$$\varLambda:= \eta_{1}^{b_1}\cdots \eta_{s}^{b_s} $$ and $ B \geq \max(|b_1|, \cdots , |b_s|) $.\\
	
	Let $ A_i \geq \max(d_{\mathbb{K}}h(\eta_{i}), | \log \eta_{i}|,\; 0.16)  $ be real numbers, for $ i = 1, \cdots , s $.\\
	With the above notation, Matveev \cite{ref8} 
	proved the following result.\\ Then, assuming that
	$\varLambda\neq0$, we have
	$$\varLambda > \exp(-1.4 \times 30^{s+3} \times s^{4.5} \times d_{\mathbb{K}}^{2}(1 +\log d_{\mathbb{K}})(1 + \log B)A_1\cdots A_s).$$	 
\end{theorem}
We also need the following result from Sanchez and Luca\cite{ref10}. 
\begin{lem}(Lemma 7 of \cite{ref10}) ~\\ \label{lem2}
	If $l\geq1, H>(4l^2)^l$ and $H>L/(\log L)^l$, then $L<2^lH(\log H)^l$.
\end{lem}
\subsection{Reduction method}
The bounds on the variables obtained via Baker's theory \cite{ref5} 
are too large for any computational purposes. To reduce the bounds we use the reduction method due to Dujella and Peth\H{o} [\cite{ref7} Lemma 5a]. 
\begin{lem}
	\label{lem3}
	Let $M$ be a positive integer, $p/q$ be a convergent of the continued fraction expansion of irrational number $\tau$ such as $q>6M$, and $A, B, \mu$ be some real numbers with $A>0$ and $B>1.$ Furthermore, let $$ \varepsilon:=||\mu q||-M.||\tau q||. $$
	If $\varepsilon>0$, the there is no solution to the inequality 
	$$  0<|u\tau-v+\mu|<AB^{-w}$$
	in positive integers $u, v$ and $w$ with $$ u\leq M \text{ and } w\geq \frac{\log(Aq/\varepsilon)}{\log B}.$$
	Here, for a real number $X$, $||X||$ denotes the distance from $X$ to the nearest integer, that is, $||X||:=\min_{n\in\mathbb{Z}}|X-n|$. 
\end{lem}
\section{Proofs of main result}\label{sect4}
\subsection{Proof of Theorem   \ref{theorm1}} 
\label{subs4.1}

\begin{proof}
	In order to establish Theorem \ref{eq1}, we make use of the following lemma, which gives a bound relating the size of $k$  to the parameters $t$ and $g$.
	\begin{lem} 
		\label{lem4}
		All solutions of the Diophantine equation (\ref{eq1}) satisfy $$ k<12t\log g. $$
	\end{lem}
	
	\begin{proof}
		To proof this lemma, we combine equation (\ref{eq1}) with inequalities (\ref{eq2}). Therefore, using $\ell\leq m\leq n\leq t$ and $1\leq d_1\leq d_2\leq d_3\leq d_4\leq g-1$, we have:
		$$ \alpha^{k-2}\leq N_k=d_1\frac{g^\ell-1}{g-1}.d_2\frac{g^m-1}{g-1}.d_3\frac{g^n-1}{g-1}.d_4\frac{g^t-1}{g-1}\leq(g^t-1)^4<g^{4t}.$$
		Taking logarithm on the both sides, we get $(k-2)\log\alpha<4t\log g$.\\
		Since $t\geq2$ and $g\geq2$, we obtain $$k<\left(\frac{4}{\log\alpha}+\frac{2}{t\log} \right)t\log g < 12t\log g. \text{ This ends the proof.}  $$
	\end{proof}
	
	Note that if $t=1$, then $l=m=n=1$. So the equation (\ref{eq1}) becomes $$ N_k=d_1d_2d_3d_4,$$
	Which implies $$ \alpha^{k-2}\leq(g-1)^4<g^4.$$
	And finally we have $$k<2+4\frac{\log g}{\log\alpha}=\left(\frac{4}{\log\alpha}+\frac{2}{\log g} \right)\log g<14\log g \text{ since } g\geq2.$$
	Now, suppose that $t\geq2.$ From (\ref{eq1}) and (\ref{eq3}), we have 
	$$ N_k=a\alpha^k+b\beta^k+c\gamma^k=d_1\frac{g^\ell-1}{g-1}.d_2\frac{g^m-1}{g-1}.d_3\frac{g^n-1}{g-1}.d_4\frac{g^t-1}{g-1},$$
	which implies 
	\begin{equation}
	\label{eq5}
	\begin{split}
	a\alpha^k-\frac{d_1d_2d_3d_4g^{\ell+m+n+t}}{(g-1)^4}&= -\frac{d_1d_2d_3d_4\left(g^{\ell+m+n}+ g^{\ell+m+t}+g^{\ell+n+t}+g^{m+n+t}\right) }{(g-1)^4}\\&+\frac{d_1d_2d_3d_4\left(g^{\ell+m}+g^{\ell+n}+ g^{\ell+t}+g^{m+n}+g^{m+t}+g^{n+t}\right) }{(g-1)^4}\\&-\frac{d_1d_2d_3d_4\left(g^{\ell}+ g^{m}+g^{n}+g^{t}\right) }{(g-1)^4} +\frac{d_1d_2d_3d_4}{(g-1)^4}-\zeta_n.
	\end{split}
	\end{equation}
	Taking the absolute values of both sides of (\ref{eq5}) and using (\ref{eq4}), we get
	\begin{equation}
	\label{eq6}
	\begin{split}
	\left|a\alpha^k-\frac{d_1d_2d_3d_4g^{\ell+m+n+t}}{(g-1)^4} \right| &<\frac{d_1d_2d_3d_4\left(g^{\ell+m+n}+ g^{\ell+m+t}+g^{\ell+n+t}+g^{m+n+t}\right) }{(g-1)^4}\\&+\frac{d_1d_2d_3d_4\left(g^{\ell+m}+g^{\ell+n}+ g^{\ell+t}+g^{m+n}+g^{m+t}+d_1d_2d_3d_4g^{n+t}\right) }{(g-1)^4}\\&+\frac{d_1d_2d_3d_4\left(g^{\ell}+ g^{m}+g^{n}+g^{t}\right) }{(g-1)^4} +\frac{d_1d_2d_3d_4}{(g-1)^4}+\frac{1}{\alpha^{n/2}}.
	\end{split}
	\end{equation}
	When we multiply both sides of the above inequality by
	$\frac{(g-1)^4}{d_1d_2d_3d_4g^{\ell+m+n+t}}$ and using the fact $\ell\leq m\leq n\leq t$, we have
	\begin{equation*}
	\begin{split}
	\left|\frac{(g-1)^4a\alpha^k g^{-(\ell+m+n+t)}}{d_1d_2d_3d_4}-1 \right| &<\frac{1}{g^{\ell}}+\frac{1 }{g^m}+\frac{1}{g^{n}}+\frac{1}{g^{t}}\\&+\frac{1}{g^{\ell+m}}+\frac{1 }{g^{\ell+n}}+\frac{1}{g^{\ell+t}}+\frac{1}{g^{m+n}}+\frac{1}{g^{m+t}}+\frac{1}{g^{n+t}}\\&+\frac{1}{g^{\ell+m+n}}+\frac{1 }{g^{\ell+m+t}}+\frac{1}{g^{\ell+n+t}}+\frac{1}{g^{m+n+t}}\\&+
	\frac{1}{g^{\ell+m+n+t}}+\frac{(g-1)^4}{d_1d_2d_3d_4g^{\ell+m+n+t}\alpha^{n/2}}\\&<\frac{16}{g^\ell}.
	\end{split}
	\end{equation*}
	
	Which implies
	\begin{equation}
	\label{eq7}
	\left|\frac{a(g-1)^4}{d_1d_2d_3d_4}\cdot \alpha^k \cdot g^{-(\ell+m+n+t)}-1 \right| <\frac{16}{g^\ell}.
	\end{equation}
	
	So we can take $$ \varLambda_1:= \frac{a(g-1)^4}{d_1d_2d_3d_4}\cdot \alpha^k\cdot g^{-(\ell+m+n+t)}-1.$$
	
	Let us show that $\varLambda_1\neq0$. We proceed by contrary.To that end, let us assume that $\varLambda_1=0$. Then
	$$ a\alpha=\frac{d_1d_2d_3d_4}{(g-1)^4}\cdot g^{\ell+m+n+t}, $$
	which implies
	$$\sigma_f\left( a\alpha^k\right)=b\beta^k=\frac{d_1d_2d_3d_4}{(g-1)^4}\cdot g^{\ell+m+n+t}. $$
	Taking the absolute value, we get 
	$$   \left| b\beta^k\right|=\left|\frac{d_1d_2d_3d_4}{(g-1)^4}\cdot g^{\ell+m+n+t} \right|.  $$
	
	We have $\left| b\beta^k\right|<1$ instead $\left|\frac{d_1d_2d_3d_4}{(g-1)^4}\cdot g^{\ell+m+n+t} \right| >1$ since $1\leq \ell\leq m\leq n \leq t,$ which leads to a contradiction. Hence $\varLambda_1 \neq0$.\\
	With a view to applying Matveev's result to $\varLambda_1$, let us set $$ s:=3,\; \eta_1:=\frac{a(g-1)^4}{d_1d_2d_3d_4},\; \eta_2:=\alpha,\; \eta_3:=g,$$
	$$b_1:=1,\; b_2:= k,\; b_3:= -(\ell+m+n+t),$$
	and $\mathbb{K}=\mathbb{Q}(\eta_1, \eta_2,\eta_3)=\mathbb{Q}(\alpha)$, which is a real number field of degree $d_{\mathbb{K}}=3$.\\
	Using the properties of the logarithmic height, it holds that, 
	$$ h(\eta_2)=h(\alpha)=\frac{\log\alpha}{3}, \; h(\eta_3)=h(g)=\log g$$ and 
	\begin{center}
		\begin{equation*}
		\begin{split}
		h(\eta_1)&=h\left( \frac{a(g-1)^4}{d_1d_2d_3d_4}\right)\\
		&\leq h(a)+h\left(\frac{(g-1)^4}{d_1d_2d_3d_4} \right) \\
		&\leq\frac{1}{3}\log31+\log\left(\max\{(g-1)^4, d_1d_2d_3d_4\} \right) \\
		&<2+4\log g=\left(4+\frac{2}{\log g} \right) \log g<7\log g \text{ since } g\geq 2.
		\end{split}
		\end{equation*}
	\end{center}
	Thus we can take $$ A_1=21\log g,\; A_2:=\log\alpha \text{ and } A_3:=3\log g.$$
	Using Lemma \ref{lem4} and $\ell\leq m\leq n\leq t$, we have $$\max\{|b_1|, |b_2|, |b_3|\}=\max\{1, k, \ell+m+n+t\}<12t\log g,$$ so we can  put $B=12t\log g$. Using Theorem \ref{theorem3}, we obtain
	\begin{eqnarray*}
		\log|\varLambda_1|&>&-1.4\times 30^6\times3^{4.5}\times3^2.(1+\log3).(1+\log(12t\log g))\\ && \times(21\log g)(\log g\alpha)(3\log g).\\
		&>&-6.5\times10^{13}(1+\log(12t\log g))(\log^2g).
	\end{eqnarray*}
	Comparing the above inequality with (\ref{eq7}) we obtain that $$ \ell\log g-\log16<6.5\times\times10^{13}(1+\log(12t\log g))(\log^2g).$$
	Since $g\geq 2$ and $t\geq2$, we have $$ 1+\log(12t\log g)<12\log t\log g.$$
	So we have $$\ell<7.8\times10^{14}\log t\log^2g.$$
	Using  (\ref{eq1}) and (\ref{eq2}) once again, and reorganizing the computations, we obtain
	\begin{eqnarray*}
		\frac{a\alpha^k(g-1)}{d_1(g^\ell-1)}+\frac{\zeta_k(g-1)}{d_1(g^\ell-1)}&=&\frac{d_2d_3d_4}{(g-1)^3}(g^{m+n+t}-g^{m+n}-g^{m+t}\\&&-g^{n+t}+g^m+g^n+g^t+1).
	\end{eqnarray*}
	
	Which implies 
	\begin{equation}
	\label{eq8}
	\begin{split}
	\frac{a\alpha^k(g-1)}{d_1(g^\ell-1)}-\frac{d_2d_3d_4g^{m+n+t}}{(g-1)^3}&=\frac{\zeta_k(g-1)}{d_1(g^\ell-1)}-\frac{d_2d_3d_4g^{m+n}}{(g-1)^3}\\&-\frac{d_2d_3d_4g^{m+t}}{(g-1)^3}-\frac{d_2d_3d_4g^{n+t}}{(g-1)^3}+\frac{d_2d_3d_4g^{m}}{(g-1)^3}\\&+\frac{d_2d_3d_4g^{n}}{(g-1)^3}+\frac{d_2d_3d_4g^{t}}{(g-1)^3}-\frac{d_2d_3d_4}{(g-1)^3}.
	\end{split}
	\end{equation}
	Taking the absolute value of the both sides of ($\ref{eq8}$), we have
	\begin{eqnarray*}
		\left| \frac{a\alpha^k(g-1)}{d_1(g^\ell-1)}-\frac{d_2d_3d_4g^{m+n+t}}{(g-1)^3}\right|&<& \frac{(g-1)}{d_1(g^\ell-1)\alpha^{k/2}}+\frac{d_2d_3d_4g^{m+n}}{(g-1)^3}+\frac{d_2d_3d_4g^{m+t}}{(g-1)^3}\\&&+\frac{d_2d_3d_4g^{n+t}}{(g-1)^3} +\frac{d_2d_3d_4g^{m}}{(g-1)^3}+\frac{d_2d_3d_4g^{n}}{(g-1)^3}\\&&+\frac{d_2d_3d_4g^{t}}{(g-1)^3}+\frac{d_2d_3d_4}{(g-1)^3}. 
	\end{eqnarray*}
	Multiplying both sides of the inequality above by $\frac{(g-1)^3}{d_2d_3d_4g^{m+n+t}}$ and using the fact that $2\leq m\leq n \leq t$, we get
	\begin{eqnarray*}
		\left|\frac{a(g-1)^4}{d_1d_2d_3d_4(g^\ell-1)}\cdot \alpha^k\cdot g^{-(m+n+t)}-1 \right| &<& \frac{(g-1)^4}{d_1d_2d_3d_4(g^\ell-1)\alpha^{k/2}g^{m+n+t}}+\frac{1}{g^{m}}+\frac{1}{g^{n}}+\frac{1}{g^{t}}\\&&+\frac{1}{g^{m+n}}+\frac{1}{g^{m+t}}+\frac{1}{g^{n+t}}+\frac{1}{g^{m+n+t}}<\frac{8}{g^m}.
	\end{eqnarray*}
	Then we have
	\begin{equation}
	\label{eq9}
	\left|\frac{a(g-1)^4}{d_1d_2d_3d_4(g^\ell-1)}\cdot \alpha^k\cdot g^{-(m+n+t)}-1 \right|<\frac{8}{g^m}.
	\end{equation}
	Now, set $$ \varLambda_2:=\frac{a(g-1)^4}{d_1d_2d_3d_4(g^\ell-1)}\cdot \alpha^k\cdot g^{-(m+n+t)}-1. $$
	In a similar way, one can easily check that $\varLambda_2\neq0$, proceeding as we did for $\varLambda_1$. Let us apply Matveev's result for $\varLambda_2$.
	Let $$  s:=3,\; \eta_1:=\frac{a(g-1)^4}{d_1d_2d_3d_4(g^\ell-1)},\; \eta_2:=\alpha,\; \eta_3:= g,$$
	$$ b_1:=1,\; b_2:=k,\; b_3:= -(m+n+t)$$
	and $\mathbb{K}:=\mathbb{Q}(\eta_1, \eta_2, \eta_3)=\mathbb{Q}(\alpha)$ of degree $d_{\mathbb{K}}=3$. By Lemma \ref{lem4}  we have $k<12t\log g$, so we put $B:=12t\log g$. We have $$ h(\eta_2)=h(\alpha)=\frac{\log\alpha}{3}, \; h(\eta_3)=h(g)=\log g$$
	and \begin{eqnarray*}
		h(\eta_1)&=&h\left(\frac{a(g-1)^4}{d_1d_2d_3d_4(g^\ell-1)}\right)\\
		&\leq& h(a)+h\left(\frac{(g-1)^4}{d_1d_2d_3d_4(g^\ell-1)}\right)\\
		&\leq& \frac{1}{3}\log31+\log\left(\max\{(g-1)^4, d_1d_2d_3d_4\} \right)+h\left(\frac{1}{g^\ell-1} \right) \\
		&<& 2+ 4\log(g-1)+\log(g^\ell-1)\\
		&<&(4+\ell)\log g+2\\
		&<&(7+\ell)\log g \text{ since } g \geq 2.
	\end{eqnarray*}
	Hence, since $\ell<7.8\times10^{14}\log t\log^2g$, we have
	$$
	h(\eta_1)<7.8\times10^{14}\log t\log^3g.
	$$
	Thus, we can take
	$$ A_1:=2.34\times10^{15}\log t\log^3g,\; A_2:=\log\alpha,\; \text{ and } A_3:=3\log g. $$
	Using Theorem \ref{theorem3} we see that 
	\begin{eqnarray*}
		\log|\varLambda_2|&>&-1.4\times30^6\times3^{4.5}\times3^2(1+\log3)(1+\log(12t\log g))\\
		&& \times(2.34\times10^{15}\log t\log^3g)(\log\alpha)(3\log g)\\
		&>& -7.26\times10^{27}(1+\log(12t\log g))\log t\log^4g.
	\end{eqnarray*}
	Comparing with (\ref{eq9}), we get $$ m\log g-\log8<7.26\times10^{27}(1+\log(12t\log g))\log t\log^4g.$$
	And using the fact that $$1+\log(12t\log g)<12\log t\log g,$$
	we have $$m<8.7\times10^{28}\log^2t\log^4g.$$
	
	In the order to find an upper bound of $n$ in terms of $g \text{ and } t$ we rewrite (\ref{eq1}), using (\ref{eq2}), so we get
	\begin{eqnarray*}
		\frac{a\alpha^k(g-1)^2}{d_1d_2(g^\ell-1)(g^m-1)}+\frac{\zeta_k(g-1)^2}{d_1d_2(g^\ell-1)(g^m-1)}&=&\frac{d_3d_4(g^n-1)(g^t-1)}{(g-1)^2}\\
		&=& \frac{d_3d_4g^{n+t}}{(g-1)^2}-\frac{d_3d_4g^{n}}{(g-1)^2}\\&&-\frac{d_3d_4g^{t}}{(g-1)^2}+\frac{d_3d_4}{(g-1)^2}.
	\end{eqnarray*}
	Then we have
	\begin{equation}
	\label{eq10}
	\begin{split}
	\frac{a\alpha^k(g-1)^2}{d_1d_2(g^\ell-1)(g^m-1)}-\frac{d_3d_4g^{n+t}}{(g-1)^2}&=-\frac{\zeta_k(g-1)^2}{d_1d_2(g^\ell-1)(g^m-1)}-\frac{d_3d_4g^{n}}{(g-1)^2}\\&-\frac{d_3d_4g^{t}}{(g-1)^2}+\frac{d_3d_4}{(g-1)^2}.
	\end{split}
	\end{equation} 
	Taking the absolute values of the both sides of (\ref{eq10})  and using (\ref{eq4}), we get 
	\begin{equation*}
	\begin{split}
	\left| \frac{a\alpha^k(g-1)^2}{d_1d_2(g^\ell-1)(g^m-1)}-\frac{d_3d_4g^{n+t}}{(g-1)^2}\right|&\leq \frac{(g-1)^2}{d_1d_2(g^\ell-1)(g^m-1)\alpha^{k/2}}\\&+\frac{d_3d_4g^{n}}{(g-1)^2}+\frac{d_3d_4g^{t}}{(g-1)^2}+\frac{d_3d_4}{(g-1)^2}.
	\end{split}
	\end{equation*}
	Multiplying both sides of inequality above by $\frac{(g-1)^2}{d_3d_4g^{n+t}}$ and noticing the fact that\\ $2\leq n\leq t$, we get,
	\begin{eqnarray*}
		\left| \frac{a(g-1)^4}{d_1d_2d_3d_4(g^\ell-1)(g^m-1)}\cdot \alpha^k\cdot g^{-(n+t)}-1\right| &<& \frac{(g-1)^4}{d_1d_2d_3d_4(g^\ell-1)(g^m-1)\alpha^{k/2}g^{n+t}}\\&&+\frac{1}{g^n}+\frac{1}{g^t}+\frac{1}{g^{n+t}}\\
		&<& \frac{4}{g^n}.
	\end{eqnarray*}
	So we get 
	\begin{equation}
	\label{eq11}
	\left| \frac{a(g-1)^4}{d_1d_2d_3d_4(g^\ell-1)(g^m-1)}\cdot \alpha^k\cdot g^{-(n+t)}-1\right|<\frac{4}{g^n}.
	\end{equation}
	We put $$\varLambda_3=\frac{a(g-1)^4}{d_1d_2d_3d_4(g^\ell-1)(g^m-1)}\cdot \alpha^k\cdot g^{-(n+t)}-1.$$
	One can verify that $\varLambda_3\neq0$. Let us apply once again Matveev's result for $\varLambda_3$. Let
	$$ s:=3,\;\; \eta_1= \frac{a(g-1)^4}{d_1d_2d_3d_4(g^\ell-1)(g^m-1)},\;\; \eta_2:=\alpha,\;\; \text{ and } \eta_3:= g,$$
	$$ b_1:=1, \; b_2:=k, \; b_3=-(n+t) $$
	and $\mathbb{K}:=\mathbb{Q}(\eta_1, \eta_2, \eta_3)=\mathbb{Q}(\alpha)$ of degree $d_{\mathbb{K}}=3$. By using Lemma \ref{lem4} we have $k<12t\log g$, so we can take $B=12t\log g$. We have
	$$ h(\eta_2)=h(\alpha)=\frac{\log\alpha}{3},\; h(\eta_3)=h(g)=\log g $$
	and 
	\begin{eqnarray*}
		h(\eta_1)&=&h\left(\frac{a(g-1)^4}{d_1d_2d_3d_4(g^\ell-1)(g^m-1)} \right)\leq h(a)+h\left( \frac{(g-1)^4}{d_1d_2d_3d_4(g^\ell-1)(g^m-1)}\right)  \\
		&\leq&\frac{1}{3}\log31+\log\left(\max\{(g-1)^4, d_1d_2d_3d_4\} \right) +h\left(\frac{1}{g^\ell-1} \right) +h\left(\frac{1}{g^m-1} \right)\\
		&<&2+4\log(g-1)+\log(g^\ell-1)+\log(g^m-1)\\
		&<&(4+\ell+m)\log g+2=\left(4+\ell+m+\frac{2}{\log g} \right)\log g <(7+\ell+m)\log g\\&&\text{ since } g\geq2.
	\end{eqnarray*}
	Using the fact that $\ell< 7.8\times10^{14}\log t\log^2g$ and $ m<8.7\times10^{28}\log^2t\log^4g$, we get $$h(\eta_1)<8.7\times10^{28}\log^2t\log^5g .$$
	We take $$A_1:=2.61\times10^{28}\log^2t\log^5g, \; A_2:=\log\alpha \text{ and } A_3=3\log g .$$
	Using Theorem \ref{theorem3} we see that
	\begin{eqnarray*}
		\log|\varLambda_3|&>&-1.4\times30^6\times3^{4.5}\times3^2(1+\log3)(1+\log(12t\log g))\\ &&\times(2.61\times10^{28}\log^2t\log^5g)(3\log g\log \alpha)\\
		&>&-8.1\times10^{41} (1+\log(12t\log g))\log^2t\log^6g.
	\end{eqnarray*}
	Comparing with (\ref{eq11}), we get $$n\log g-\log4 <8.1\times10^{41} (1+\log(12t\log g))\log^2t\log^6g.$$
	We have 
	$$1+\log(12t\log g)<12\log t\log g,$$
	Thus 
	$$
	n<9.72\times10^{42}\log^3t\log^6g.
	$$
	
	
	This step will mark the end of the proof of Theorem \ref{theorm1}. To get there we need to rearrange (\ref{eq1}) in the following form in view to apply the Theorem \ref{theorem3}:
	\begin{eqnarray*}
		\frac{a\alpha^k(g-1)^3}{d_1d_2d_3(g^\ell-1)(g^m-1)(g^n-1)}+\frac{\zeta_k(g-1)^3}{d_1d_2d_3(g^\ell-1)(g^m-1)(g^n-1)}&=&\frac{d_4g^t}{(g-1)}-\frac{d_4}{(g-1)}.
	\end{eqnarray*}
	Then we have
	\begin{equation}
	\label{eq12}
	\begin{split}
	\frac{a\alpha^k(g-1)^3}{d_1d_2d_3(g^\ell-1)(g^m-1)(g^n-1)}-\frac{d_4g^t}{(g-1)}&=-\frac{\zeta_k(g-1)^3}{d_1d_2d_3(g^\ell-1)(g^m-1)(g^n-1)}-\frac{d_4}{(g-1)}.
	\end{split}
	\end{equation} 
	Taking the absolute values of the both sides of (\ref{eq12})  and using (\ref{eq4}), we get 
	\begin{equation*}
	\begin{split}
	\left| \frac{a\alpha^k(g-1)^3}{d_1d_2d_3(g^\ell-1)(g^m-1)(g^n-1)}-\frac{d_4g^t}{(g-1)}\right|&\leq \frac{(g-1)^3}{d_1d_2d_3(g^\ell-1)(g^m-1)(g^n-1)\alpha^{k/2}}+\frac{d_4}{(g-1)}.
	\end{split}
	\end{equation*}
	Multiplying both sides of inequality above by $\frac{g-1}{d_4g^{t}}$ and noticing the fact that $t\geq2$, we get,
	
	\begin{equation}
	\label{eq13}
	\left| \frac{a(g-1)^4}{d_1d_2d_3d_4(g^\ell-1)(g^m-1)(g^n-1)}\alpha^k\cdot g^{-t}-1\right|< \frac{2}{g^{t-1}}.
	\end{equation}
	We put $$\varLambda_4=\frac{a(g-1)^4}{d_1d_2d_3d_4(g^\ell-1)(g^m-1)(g^n-1)}\cdot \alpha^k\cdot g^{-t}-1.$$
	One can verify that $\varLambda_4\neq0$. Let us analyze Matveev's result for $\varLambda_4$. Let
	$$ s:=3,\;\; \eta_1= \frac{a(g-1)^4}{d_1d_2d_3d_4(g^\ell-1)(g^m-1)(g^n-1)},\;\; \eta_2:=\alpha,\; \text{ and } \eta_3:= g,$$
	$$ b_1:=1, \; b_2:=k, \; b_3=-(n+t) $$
	and $\mathbb{K}:=\mathbb{Q}(\eta_1, \eta_2, \eta_3)=\mathbb{Q}(\alpha)$ of degree $d_{\mathbb{K}}=3$. By using Lemma \ref{lem4} we have \\$k<12t\log g$, so we put $B=12t\log g$. We have
	$$ h(\eta_2)=h(\alpha)=\frac{\log\alpha}{3},\; h(\eta_3)=h(g)=\log g $$
	and 
	\begin{eqnarray*}
		h(\eta_1)&=&h\left(\frac{a(g-1)^4}{d_1d_2d_3d_4(g^\ell-1)(g^m-1)(g^n-1)} \right)\leq h(a)\\ &&+h\left( \frac{(g-1)^4}{d_1d_2d_3d_4(g^\ell-1)(g^m-1)(g^n-1)}\right)  \\
		&\leq&\frac{1}{3}\log31+\log\left(\max\{(g-1)^4, d_1d_2d_3d_4\} \right) +h\left(\frac{1}{g^\ell-1} \right) \\&&+h\left(\frac{1}{g^m-1} \right)+h\left(\frac{1}{g^n-1} \right)\\
		&<&2+4\log(g-1)+\log(g^\ell-1)+\log(g^m-1)+\log(g^n-1)\\
		&<&(4+\ell+m+n)\log g+2=\left(4+\ell+m+n+\frac{2}{\log} \right)\log g\\ &<&(7+\ell+m+n)\log g\text{ since } g\geq2.
	\end{eqnarray*}
	Using the fact that  $$\ell< 7.8\times10^{14}\log t\log^2g,\;\; m<8.7\times10^{28}\log^2t\log^4g \text{ and }n<9.72\times10^{42}\log^3t\log^6g,$$ we get 
	$$ h(\eta_{1})<9.72\times10^{42}\log^3t\log^7g. $$
	So we can take $$A_1:=2.92\times10^{43}\log^3t\log^7g=, \; A_2:=\log\alpha \text{ and } A_3:=3\log g. $$
	Using Theorem \ref{theorem3} we see that
	\begin{eqnarray*}
		\log|\varLambda_4|&>&-1.4\times30^6\times3^{4.5}\times3^2(1+\log3)(1+\log(12t\log g))\\ &&\times(2.92\times10^{43}\log^3t\log^7g)(3\log g\log \alpha)\\
		&>&-9.1\times10^{55}(1+\log(12t\log g))\log^3t\log^8g.
	\end{eqnarray*}
	Comparing with (\ref{eq13}), we get 
	$$
	(t-1)\log g-\log2 <9.1\times10^{55}(1+\log(12t\log g))\times\log^3t\log^8g.
	$$
	We have 
	$$1+\log(12t\log g)<12\log t\log g.$$
	So we have $$t<1.1\times10^{57}\log^4t\log^8g.$$
	To finally have the upper bound of $t$ in term of $g$, we need to apply the Lemma \ref{lem2} due to Guzm\'{a}n en Lucas.\\
	
	Taking $l:= 4$, $L:= t$ and $H:= 1.1\times10^{57}\log^8g$, we get using Lemma above that \begin{eqnarray*}
		t&<&2^4\times1.1\times10^{57}\log^8g\log^4(1.1\times10^{57}\log^8g)\\
		&<& 1.76\times10^{58}\log^8g(131.35+8\log\log g)^4\\
		t&<& 2.11\times10^{67}\log^{12}g.
	\end{eqnarray*}
	Notice that we have use the inequality $131.35+8\log\log g<186\log g$ which holds since $g\geq2$.\\
	Moreover, by Lemma \ref{lem4}, we get $$ k<2.54\times10^{68}\log^{13}g.$$ 
\end{proof}
\subsection{Proof of Theorem \ref{theorem2}}
In subsection \ref{subs4.1}, we  established that for $2\leq g\leq 12$, $$ \ell\leq m\leq n\leq t< 1.18\times10^{72}  \text{ and } k<3.5\times10^{73}.$$
The next objective is to refine the upper bounds above in order to delimit the interval containing the possible solutions of (\ref{eq1}), following a four-step approach.\\
\begin{step}
\end{step}
Referring to (\ref{eq7}), we introduce $$\Gamma_1:=\log(\varLambda_1+1)=k\log\alpha-(l+m+n+t)\log g+\log\left(\frac{a(g-1)^4}{d_1d_2d_3d_4} \right).$$
Notice that, since $\Gamma_1=\log(\varLambda_1+1)$, we have $|\mathrm{e}^{\Gamma_1}-1|=|\varLambda_1|<\frac{16}{g^l}$. \\ Observe that $\Gamma_1\neq0$, since $\varLambda_1\neq0$. So, for $\ell\geq5$ and $g\geq2$, we have $$|\mathrm{e}^{\Gamma_1}-1|<\frac{16}{g^l}<\frac{1}{2}. $$ Since $|x|<2|\mathrm{e}^x-1|$, if $|x|<\frac{1}{2}$ holds, then $$ |\Gamma_1|<2|\mathrm{e}^{\Gamma_1}-1|=2|\varLambda_1|<\frac{32}{g^\ell}.$$
Substituting $\Gamma_1$ in the above inequality with its value and dividing through
by $\log g$, we get
$$\left|\,k\left( \frac{\log\alpha}{\log g}\right)-(\ell+m+n+t)+\frac{\log\left( \frac{a(g-1)^4}{d_1d_2d_3d_4}\right) }{\log g}  \right| < \frac{32}{\log g}\cdot g^{-\ell}.$$
Then, we can apply Lemma \ref{lem3} with the data
$$ \tau:=\frac{\log\alpha}{\log g},\; \;\mu:= \frac{\log\left( \frac{a(g-1)^4}{d_1d_2d_3d_4}\right) }{\log g},\;\; A:= \frac{32}{\log g},$$
$$B:= g,\;\; w:=\ell,\;\; u=k, \text{ and } v=\ell+m+n+t,$$ with $$1\leq d_1\leq d_2\leq d_3\leq d_4\leq g-1. $$
We can take $M:=3.5\times10^{73}$, since $k<12t\log g<3.5\times10^{73}$. So, for the remaining proof, we use Mathematica to apply Lemma \ref{lem3}. For the computations, if the first convergent $ q_t $ is such that $ q_t > 6M $ does not satisfy
the condition $\varepsilon>0$, then we use the next convergent until we find the one that satisfies the conditions. Thus, we have the results given in Table \ref{table2}:
\\
\begin{center}
	\begin{table}[!h]
		\caption{Upper bound on $\ell$}
		\label{table2}
		\begin{tabular}{|l||*{11}{c|}}\hline
			$g$& $2$&$3$& $4$&$5$&$6$&$7$&$8$&$9$&$10$& $11$&$12$ \\ \hline \hline
			$q_t$& $q_{153}$&$q_{141}$&$q_{147}$&$q_{145}$&$q_{134}$&$q_{151}$&$q_{153}$&$q_{147}$&$q_{134}$&$q_{140}$&$q_{154}$\\ \hline 
			$\varepsilon\geq$& $0.34$&$0.39$&$0.051$&$0.011$&$0.0037$&$0.000019$&$0.0023$&$0.0016$&$0.006$&$0.00028$&$0.0062$\\ \hline
			$\ell\leq$&$257$&$162$& $130$&$112$&$101$&$96$&$88$&$84$&$79$&$76$&$75$\\ \hline
		\end{tabular}
	\end{table}
\end{center}
Therefore
$$ 1\leq \ell \leq \frac{\log\left((32/\log 2).q_{153}/0.34 \right) }{\log 2}\leq257.$$
\begin{step}
\end{step}
Now, we focus on locating
the real range of $ m $. To do this, let us consider $$ \Gamma_2=\log(\varLambda_2+1)=k\log\alpha-(m+n+t)+\log\left( \frac{a(g-1)^4}{d_1d_2d_3d_4(g^\ell-1)}\right).$$
Thus inequality (\ref{eq9}) becomes  
$$ |\mathrm{e}^{\Gamma_2}-1|<\frac{8}{g^m}<\frac{1}{2}. $$ Which holds for $m\geq4$. Thus,
\begin{equation}
\label{eq14}
\left|\,k\left( \frac{\log\alpha}{\log g}\right)-(m+n+t)\log g+\frac{\log\left( \frac{a(g-1)^4}{d_1d_2d_3d_4(g^{\ell}-1)}\right) }{\log g}  \right| < \frac{16}{\log g}\cdot g^{-m}.
\end{equation}
Therefore we can applied Lemma \ref{lem3} to the above inequality (\ref{eq14}) with the following data
$$ \tau:=\frac{\log\alpha}{\log g},\; \;\mu:= \frac{\log\left( \frac{a(g-1)^4}{d_1d_2d_3d_4(g^{\ell}-1)}\right) }{\log g},\;\; A:= \frac{16}{\log g},$$
$$B:= g,\;\; w:=m,\;\; u=k, \text{ and } v=m+n+t$$  with $$1\leq d_1\leq d_2\leq d_3\leq d_4\leq g-1 \text{ and } 1\leq\ell\leq265. $$
We can take $M:=3.5\times10^{73}$, since $k<12t\log g<3.5\times10^{73}.$\\ With \textit{Mathematica} we get the results given in Table \ref{table3}:\\
\begin{center}
	\begin{table}[!h]
		\caption{Upper bound on $m$}
		\label{table3}
		\begin{tabular}{|l||*{7}{c|}}\hline
			$g$& $2$&$3$& $4$&$5$&$6$&$7$&$8$ \\ \hline \hline
			$q_t$& $q_{153}$&$q_{141}$&$q_{147}$&$q_{146}$&$q_{134}$&$q_{151}$&$q_{153}$\\ \hline 
			$\varepsilon\geq$& $0.0012$&$3\times10^{-4}$&$3.3\times10^{-4}$&$6.9\times10^{-4}$&$6.6\times10^{-4}$&$1.9\times10^{-5}$&$1.3\times10^{-4}$\\ \hline
			$m\leq$&$264$&$168$& $133$&$113$&$102$&$95$&$89$\\ \hline
		\end{tabular}
		\begin{tabular}{|l||*{4}{c|}}\hline
			$g$&$9$&$10$& $11$&$12$ \\ \hline \hline
			$q_t$&$q_{147}$&$q_{134}$&$q_{140}$&$q_{154}$\\ \hline 
			$\varepsilon\geq$&$2.7\times10^{-5}$&$1.8\times10^{-5}$&$1.1\times10^{-5}$&$4.6\times10^{-6}$\\ \hline
			$m\leq$&$85$&$82$&$77$&$77$\\ \hline
		\end{tabular}
	\end{table}
\end{center}
In all cases, we can conclude that
$$ 1\leq m \leq \frac{\log\left((16/\log 2).q_{153}/0.0012 \right) }{\log 2}\leq 265.$$
\begin{step}
\end{step}
For the third application of Lemma \ref{lem3},  using the inequality (\ref{eq11}), we set

$$ \Gamma_3=\log(\varLambda_3+1)=k\log\alpha-(n+t)\log g+\log\left( \frac{a(g-1)^4}{d_1d_2d_3d_4(g^\ell-1)(g^m-1)}\right). $$
Thus inequality (\ref{eq11}) becomes  
$$
|\mathrm{e}^{\Gamma_3}-1|<\frac{4}{g^m}<\frac{1}{2}.
$$ 
Which holds for $m\geq3$. It follows that
\begin{equation}
\label{eq15}
\left|\,k\left( \frac{\log\alpha}{\log g}\right)-(n+t)+\frac{\log\left( \frac{a(g-1)^4}{d_1d_2d_3d_4(g^{\ell}-1)(g^{m}-1)}\right) }{\log g}  \right| < \frac{8}{\log g}\cdot g^{-n}.
\end{equation}
Since the conditions of Lemma \ref{lem3} are satisfied, we proceed to apply this Lemma to the inequality (\ref{eq15}) with the following data
$$ \tau:=\frac{\log\alpha}{\log g},\; \;\mu:= \frac{\log\left( \frac{a(g-1)^4}{d_1d_2d_3d_4(g^{\ell}-1)(g^m-1)}\right) }{\log g},\;\; A:= \frac{8}{\log g},$$
$$B:= g,\;\; w:=n,\;\; u=k, \text{ and } v=n+t  $$  with $$1\leq d_1\leq d_2\leq d_3\leq d_4\leq g-1,\;\; 1\leq\ell\leq257 \text{ and } 1\leq m\leq265. $$
We can take $M:=3.5\times10^{73}$, since $k<12t\log g<3.5\times10^{73}.$\\ With \textit{Mathematica} we get the results given in Table \ref{table4}:

\begin{center}
	\begin{table}[!h]
		\caption{Upper bound on $n$}
		\label{table4}
		\begin{tabular}{|l||*{7}{c|}}\hline
			$g$& $2$&$3$& $4$&$5$&$6$&$7$&$8$ \\ \hline \hline
			$q_t$& $q_{153}$&$q_{141}$&$q_{147}$&$q_{146}$&$q_{134}$&$q_{151}$&$q_{153}$\\ \hline 
			$\varepsilon\geq$& $4.28\times10^{-6}$&$1.6\times10^{-6}$&$3.39\times10^{-6}$&$2.3\times10^{-6}$&$2.1\times10^{-6}$&$1.2\times10^{-6}$&$8\times10^{-6}$\\ \hline
			$n\leq$&$270$&$171$& $135$&$115$&$103$&$95$&$89$\\ \hline
		\end{tabular}
		\begin{tabular}{|l||*{4}{c|}}\hline
			$g$&$9$&$10$& $11$&$12$ \\ \hline \hline
			$q_t$&$q_{147}$&$q_{134}$&$q_{140}$&$q_{154}$\\ \hline 
			$\varepsilon\geq$&$3.2\times10^{-7}$&$1.7\times10^{-6}$&$7.3\times10^{-7}$&$9\times10^{-7}$\\ \hline
			$n\leq$&$86$&$82$&$77$&$77$\\ \hline
		\end{tabular}
	\end{table}
\end{center}
In all cases, we can conclude that
$$ 1\leq n \leq \frac{\log\left((8/\log 2).q_{153}/4.28\times10^{-6} \right) }{\log 2}\leq271.$$
\begin{step}
\end{step}

Finally, to further reduce the bound on $t$, we set
$$ \Gamma_4=\log(\varLambda_3+1)=k\log\alpha-t\log g+\log\left( \frac{a(g-1)^4}{d_1d_2d_3d_4(g^\ell-1)(g^m-1)(g^n-1)}\right). $$
Therefore inequality (\ref{eq13}) becomes  
$$
|\mathrm{e}^{\Gamma_4}-1|<\frac{2}{g^{t-1}}<\frac{1}{2}. 
$$ 
Which holds for $t\geq3$. It follows that
\begin{equation}
\label{eq16}
\left|\,k\left( \frac{\log\alpha}{\log g}\right)-t+\frac{\log\left( \frac{a(g-1)^4}{d_1d_2d_3d_4(g^{\ell}-1)(g^{m}-1)(g^n-1)}\right) }{\log g}  \right| < \frac{4}{\log g}\cdot g^{-(t-1)}.
\end{equation}
Since the conditions of Lemma \ref{lem3} are satisfied, we proceed to apply this Lemma to the inequality (\ref{eq16}) with the following data
$$ \tau:=\frac{\log\alpha}{\log g},\; \;\mu:= \frac{\log\left( \frac{a(g-1)^4}{d_1d_2d_3d_4(g^{\ell}-1)(g^m-1)(g^n-1)}\right) }{\log g},\;\; A:= \frac{4}{\log g},$$
$$B:= g,\;\; w:=t-1,\;\; u=k, \text{ and } v=t$$ with $$1\leq d_1\leq d_2\leq d_3\leq d_4\leq g-1,\;\; 1\leq\ell\leq257, $$

$$ 
1\leq m\leq265,\,\; \text{ and }1\leq n\leq 271.
$$
We can take $M:=3.5\times10^{73}$, since $k<12t\log g<3.5\times10^{73}.$\\ With \textit{Mathematica} we get the results given in Table \ref{table5}:
\begin{center}
	\begin{table}[!h]
		\caption{Upper bound on $t$}
		\label{table5}
		\begin{tabular}{|l||*{7}{c|}}\hline
			$g$& $2$&$3$& $4$&$5$&$6$&$7$&$8$ \\ \hline \hline
			$q_t$& $q_{153}$&$q_{141}$&$q_{147}$&$q_{145}$&$q_{134}$&$q_{151}$&$q_{152}$\\ \hline 
			$\varepsilon\geq$& $4.28\times10^{-6}$&$1.6\times10^{-6}$&$1.9\times10^{-6}$&$2.3\times10^{-6}$&$5.4\times10^{-6}$&$6\times10^{-6}$&$8\times10^{-6}$\\ \hline
			$t\leq$&$269$&$170$& $135$&$115$&$103$&$94$&$88$\\ \hline
		\end{tabular}
		\begin{tabular}{|l||*{4}{c|}}\hline
			$g$&$9$&$10$& $11$&$12$ \\ \hline \hline
			$q_t$&$q_{147}$&$q_{134}$&$q_{140}$&$q_{154}$\\ \hline 
			$\varepsilon\geq$&$3.2\times10^{-6}$&$9.6\times10^{-6}$&$7.9\times10^{-7}$&$2.8\times10^{-6}$\\ \hline
			$t\leq$&$85$&$80$&$77$&$76$\\ \hline
		\end{tabular}
	\end{table}
\end{center}
In all cases, we can conclude that
$$ 1\leq t \leq \frac{\log\left((2/\log 2).q_{153}/3.8\times10^{-7} \right) }{\log 2}\leq 270$$
which is valid for all $ g $ such as $ 2\leq g\leq12 $. In light of the above results, we
need to check the equation (\ref{eq1}) in the cases $ 2\leq g\leq12$ for $ 1\leq d1; d2; d3;d_4\leq 11,$  $1\leq \ell\leq257$,
$1\leq m\leq265$,$1\leq n\leq271$, $1\leq t\leq270$ and $1\leq k\leq8051$. A quick inspection using \textit{Mathematica} reveals that the Diophantine equation (\ref{eq1}) in the
range $2\leq g \leq12$ has only the solutions listed in the statement of Theorem \ref{theorem2}.
This completes the proof of Theorem \ref{theorem2}.


\begin{thebibliography}{99}
	\bibitem{ref1} G. Abou-Elela, A. Elsonbaty, and M. Anwar, Two problems on Narayana numbers and repeated digit numbers, \textit{arXiv:2303.01718v2 [math.NT]} (2023), 14 pp.
	\bibitem{ref2} K. N. Ad\'edji, Balancing numbers which are products of three repdigits in base b.\textit{ Bol.
		Soc. Mat. Mex.} 29, (2023),15 pp.  
	\bibitem{ref3} K. N. Ad\'edji, A. Filipin, and A. Togb\'e, Fibonacci and Lucas numbers as products of three repdigits in base g, \textit{Rend. Circ. Mat. Palermo} (2023), 19 pp.
	\bibitem{ref4} J. P. Allouche and T. Johnson, Narayana's cows and delayed morphisms, J. d'Inform. Music. (1996), île de Tatihou, \url{https://hal.archives-ouvertes.fr/hal-02986050}.
	\bibitem{ref5}A. Baker and H. Davenport, The equations $3x^2-2=y^2$  and $8x^2-7=z^2$,\textit{ Q. J.	Math}. 20 (1969), 129-137.
	\bibitem{ref6} K. Bhoi, B. K. Patel, and P. K. Ray, Narayana numbers as sums of two base b repdigits, \textit{Acta Comment. Univ. Tartu. Math.} 26 (2022), 183-192.
	\bibitem{ref7} A. Dujella and A. Peth\H{o}, A generalization of a theorem of Baker and Davenport, \textit{Q. J. Math.} 49 (1998), 291-306.
	\bibitem{ref8} E. M. Matveev, An explicit lower bound for a homogeneous rational linear form in logarithms of algebraic numbers II, \textit{Izv. Math.} 64 (2000), 1217-1269.
	\bibitem{ref9} Y. Bugeau, M. Mignotte, and S. Siksek, Classic and modular approaches to exponential Diophantine equations I. Fibonacci and Lucas perfect powers, \textit{ Ann. of Math. J. Math.} 163(2006), 969-1018
	\bibitem{ref10} S. G. Sanchez and F. Luca, Linear combinations of factorials and S-units in a binary recurrence sequence, \textit{Ann. Math.} Qu\'e. 38 (2014), 169-188.
	\bibitem{ref11} K. Ben Yakkou, K. R. Kakanou, and P. Tiebekabe, Narayana numbers as products of three repdigits in base $g$, \textit{Acta et Commentationes Universitatis Tartuensis de Mathematica,} 27(2), (2023), 319-334. Available online  at: \url{https://ojs.utlib.ee/index.php/ACUTM}.
	\bibitem{ref12} N. J. A. Sloane et al., The On-Line Encyclopedia of Integers Sequence, 2023.
	\bibitem{ref13} P. Tiebekabe and K. N. Ad\'edji, Narayana's cows numbers which are concatenations of three repdigits in base $\varrho$, \textit{arXiv:2304.00773v1 [math.NT]} (2023), 15 pp.
	\bibitem{ref14}	Adédji, K.N. On Mulatu and generalized Lucas numbers which are products of four b-repdigits with a consequence. \textit{Indian J Pure Appl Math} (2025). \url{https://doi.org/10.1007/s13226-025-00758-w}. 
	\bibitem{ref15} J. L. Ram\'{\i}rez and V. F. Sirvent, A note on the k-narayana sequence, \textit{Ann. Math. Inform.}
	45 (2015), 91–105.
\end{thebibliography}

\end{document}